\documentclass[a4paper,12pt]{amsart}
\usepackage[ngerman,english]{babel}

\usepackage[utf8]{inputenc}
\usepackage{amssymb}
\usepackage{amsmath}
\usepackage{amsfonts}
\usepackage{amstext}
\usepackage{amsthm}
\usepackage{lmodern}

\usepackage{todonotes}

\usepackage[T1]{fontenc}
\usepackage{textcomp}
\usepackage{geometry}
\usepackage[font=footnotesize]{caption}
\usepackage{enumerate}
\usepackage{hyperref}

\newcommand{\q}{Q}
\newcommand{\qp}{Q'}

\newcommand{\IR}{{\mathbb{R}}}
\newcommand{\R}{{\mathbb{R}}}

\newcommand{\IC}{{\mathbb{C}}}

\newcommand{\cE}{{\mathcal{E}}}
\newcommand{\edv}{{\mathcal{E}^{(D)}_V}}

\newcommand{\env}{{\mathcal{E}^{(N)}_V}}

\newcommand{\QM}{{Q^{(M)}}}
\newcommand{\QK}{{Q^{(k)}}}

\newcommand{\capa}{{\operatorname{Cap}}}

\newcommand{\supp}{\operatorname{supp}}
\renewcommand{\phi}{\varphi}
\renewcommand{\epsilon}{\varepsilon}

\newcommand{\K}{\mathcal{K}}

\newcommand{\ca}{\operatorname{cap}}

\newcommand{\abs}[1]{\lvert#1\rvert}
\newcommand{\norm}[1]{\lVert#1\rVert}

\newtheorem{theorem}{Theorem}[section]
\newtheorem{lemma}[theorem]{Lemma}
\newtheorem{proposition}[theorem]{Proposition}
\newtheorem{corollary}[theorem]{Corollary}
\newtheorem*{theorem*}{Theorem}

\theoremstyle{definition}
\newtheorem{definition}[theorem]{Definition}

\newtheorem{remarks}[theorem]{Remark}
\newtheorem{example}[theorem]{Example}

\newcommand{\Hmm}[1]{\leavevmode{\marginpar{\tiny%
			$\hbox to 0mm{\hspace*{-0.5mm}$\leftarrow$\hss}%
			\vcenter{\vrule depth 0.1mm height 0.1mm width \the\marginparwidth}%
			\hbox to 0mm{\hss$\rightarrow$\hspace*{-0.5mm}}$\\\relax\raggedright #1}}}

\title[Intermediate Dirichlet forms]{Boundary representations of intermediate forms between a regular Dirichlet form and its active main part}

\author[M. Keller]{Matthias Keller}
\address{M.Keller, Institut für Mathematik, Universität Potsdam, Campus Golm, Haus 9, Karl-Liebknecht-Str. 24-25, 14476 Potsdam OT Golm, Germany}
\email{matthias.keller@uni-potsdam.de}

\author[D. Lenz]{Daniel Lenz}
\address{D.Lenz, Institut für Mathematik, Friedrich-Schiller-Universität Jena, 07737 Jena, Germany}
\email{daniel.lenz@uni-jena.de}

\author[M. Schmidt]{Marcel Schmidt}
\address{M. Schmidt, Mathematisches Institut, Universität Leipzig, Augustusplatz 10, 04109 Leipzig, Germany}
\email{marcel.schmidt@math.uni-leipzig.de}

\author[M. Schwarz]{Michael Schwarz}
\address{M. Schwarz, dotSource GmbH, Goethestr. 1, 07743 Jena, Germany}
\email{m.schwarz@dotSource.de}

\author[M. Wirth]{Melchior Wirth}
\address{M. Wirth,  Institute of Science and Technology Austria (ISTA), Am Campus 1, 3400 Klosterneuburg, Austria}
\email{melchior.wirth@ist.ac.at}

\begin{document}

\maketitle

\begin{abstract}
 We characterize all semigroups sandwiched between the semigroup of a Dirichlet form and the semigroup of its active main part. In case the Dirichlet form is regular, we give a more explicit description of the quadratic forms of the sandwiched semigroups in terms of pairs consisting of an open set and a measure on an abstract boundary. 
 \end{abstract}

\section*{Introduction}\label{chapter:boundary_rep_spec}

One prime example of different self-adjoint realizations of the same differential expression are the Dirichlet and Neumann Laplacian on a bounded domain, i.e. two operators that only differ by the choice of boundary conditions. More generally one may ask which self-adjoint realizations of a differential expression arise from choosing boundary conditions.

For the Laplacian, one possible answer was given by Arendt and Warma in \cite{AW03}: If $\Omega$ is a domain with Lipschitz boundary, a self-adjoint positive operator $L$ on $L^2(\Omega)$ is a Laplacian with Robin-type boundary conditions if and only if the associated semigroup $(e^{-tL})$ is sandwiched between the Dirichlet and Neumann heat semigroup in the sense that 
$$e^{t \Delta^{(D)}} f \leq e^{-tL} f \leq e^{t \Delta^{(N)}}f$$
for all $f \geq 0$ and $t > 0$. Here the Laplacians with Robin-type boundary conditions can best be described in terms of associated quadratic forms: The Dirichlet form $Q$ associated with $L$ satisfies $D(Q)   =  \{f \in H^1(\Omega) \mid f = 0 \text{ quasi everywhere  on  }  \Omega\setminus O\}$ and 
$$Q(f) = \int_\Omega |\nabla f|^2 dx + \int_{\partial \Omega} |\tilde f|^2 d\mu $$
for some open $O \subseteq \partial \Omega$ and a measure $\mu$ on $\partial \Omega$ not charging sets of capacity zero. Here $\tilde f$ denotes a quasi-continuous modification of $f$.

Note that in the original work of Arendt and Warma there was an additional condition that $L$ be local, but this was later shown to be superfluous by Akhlil \cite{AK18}.

This result has been generalized in several directions. Chill and Warma \cite{CW12} gave a similar characterization of (nonlinear) semigroups sandwiched between the semigroup generated by the $p$-Laplacian with Dirichlet boundary conditions and the $p$-Laplacian with Neumann boundary conditions. Later this characterization was extended to semigroups associated with local nonlinear Dirichlet forms by Claus \cite{Cla21}. In \cite{ACD21} Arora, Chill and Djida do not study sandwiched semigroups but give a characterization of all semigroups dominating a semigroup induced by a regular form.   

A related problem was studied by Posilicano in \cite{Pos15}. For a bounded domain with smooth boundary he characterizes all self-adjoint realizations of the Laplacian that generate Markovian semigroups via certain Dirichlet forms on the boundary of the domain. Applying his findings to realizations with sandwiched semigroups one obtains the result of Arendt and Warma  under higher regularity assumptions on the boundary of the domain.  For discrete Laplacians associated with infinite graphs similar characterizations of Markovian realizations were obtained by the first four authors in \cite{KLSS20}.  In this case, the employed boundary is the Royden boundary of the graph, which is defined using Gelfand theory. 

In this article we treat the question of sandwiched semigroups in the abstract context of Dirichlet forms. We start with a regular Dirichlet form without killing whose generator we take as an abstract analogue of the Dirichlet Laplacian. In this setting there is a natural analogue of the Neumann Laplacian, namely the generator of the active main part of our given regular Dirichlet form, which was introduced in \cite{S17,S18}. Our framework includes not only the Laplacian on domains treated by Arendt and Warma, but also various Laplace-like operators like fractional Laplacians, Laplacians on manifolds and metric measure spaces or Laplacians on weighted graphs and quantum graphs.

We first give an abstract characterization of the generators of semigroups that are sandwiched between the semigroup associated with a regular Dirichlet form and the semigroup associated with its active main part in terms of order properties.

To connect these sandwiched semigroups to boundary conditions, the first problem is to find a good notion of boundary in this setting. As all the quadratic forms involved are defined on the $L^2$-space of some abstract topological measure space, there is no immediate geometric notion of boundary available. As in \cite{KLSS20} and \cite{ACD21} we introduce a notion of boundary that is defined using Gelfand theory and depends on the given regular Dirichlet form.  

With this notion of boundary, we can prove an abstract version of the main result of Arendt and Warma (Theorem~\ref{theorem:main}):

\begin{theorem*}
Let $Q$ be a regular Dirichlet form on $L^2(X,m)$ without killing and $Q^{(M)}$ its active main part. For a Dirichlet form $Q'$ on $L^2(X,m)$, the following assertions are equivalent:
\begin{enumerate}[(i)]
\item There exists an open subset $O$ of $X\cup\partial X$ and a measure $\mu$ on $O\cap \partial X$ that does not charge polar sets such that $Q'$ is the closure of the quadratic form $Q^c_{O,\mu}$ given by $D(Q^c_{O,\mu})=D(Q)\cap C_c(O)$ and
\begin{equation*}
Q^c_{O,\mu}(f)=Q(f)+\int_{O\cap \partial X} f^2\,d\mu.
\end{equation*}
\item The semigroup associated with $Q'$ is sandwiched between the semigroup associated with $Q$ and the semigroup associated with $Q^{(M)}$, and $D(Q')\cap C_c(X\cup \partial X)$ is a form core for $Q'$.
\end{enumerate}
\end{theorem*}

In other words, the Dirichlet forms sandwiched between $Q$ and $Q^{(M)}$ (in the sense of domination of semigroups) are parametrized by measures on open subsets of an abstract boundary. 

In spirit our main result for regular Dirichlet forms is similar to the one of \cite{ACD21}, which treats an even more general setting without assuming the Markov property. The main differences are that for the first abstract part we need not assume any regularity of the forms and when we assume regularity, our results are more explicit.

The article is organized as follows: In Section~\ref{sec:basics} we introduce the notation used throughout this article and recall some basic facts about Dirichlet forms and domination of semigroups. In Section~\ref{sec:sandwiched_forms_abstract} we review the active main part of a regular Dirichlet form and give an abstract characterization of the Dirichlet forms sandwiched between the given regular Dirichlet form and its active main part (Theorem~\ref{theorem:abstract characterization}). In Section \ref{section:domination regular} we study some properties of the forms $Q^c_{O,\mu}$ in the main theorem stated above, in particular their closability. In Section~\ref{sec:boundary_reps} we introduce our notion of boundary and show how sandwiched Dirichlet forms can be represented by measures on open subsets of the boundary (Theorem \ref{theorem:main}). Finally, in the appendix we collect some facts about bilinear forms on spaces of compactly supported continuous functions.

Parts of this paper are based on the PhD thesis of the fourth-named author \cite{Sch20}. 

\subsection*{Acknowledgments} The first three authors acknowledge financial support of the DFG  within the priority programme  Geometry at Infinity.   M.W. acknowledges financial support by the German Academic Scholarship Foundation, by the Austrian Science Fund (FWF) through grant number F65 and the Esprit Programme [ESP 156], and by the European Research Council (ERC) under the European Union's Horizon 2020 research and innovation programme (grant agreement No 716117). For the purpose of Open Access, the authors have applied a CC BY public copyright licence to any Author Accepted Manuscript (AAM) version arising from this submission.

\section{Dirichlet forms and domination of associated semigroups}\label{sec:basics}

In this section we introduce notation and review some basic definitions and results about Dirichlet forms and domination of the associated semigroups. Unless stated otherwise, all functions are real-valued. Throughout $(X,\mathcal{A},m)$ is a $\sigma$-finite measure space and $Q$ denotes a nonnegative quadratic form with domain $ D(Q) \subseteq L^2(X,m)$.  We tacitly identify $Q$ and the bilinear form it induces by polarization. In particular, we have the convention $Q(f) = Q(f,f)$ for $f \in D(Q)$. The {\em form norm} $\norm{\cdot}_Q$ is the norm on $D(Q)$ defined by
$$\norm{f}^2_Q = Q(f) + \norm{f}^2,$$ 
where $\norm{\cdot}$ is the norm on $L^2(X,m)$. If $Q'$ is another quadratic form we write $Q \sqsubseteq Q'$ if $D(Q) \subseteq D(Q')$ and $Q(f)\geq Q'(f)$ for all $f \in D(Q)$. The induced order relation $\sqsubseteq$ on all quadratic forms is called the {\em natural order}.
%

We say that a quadratic form is {\em positive} if $Q(f,g) \geq 0$  for all nonnegative $f,g \in D(Q)$.  It is called {\em local} if $fg = 0$ implies $Q(f,g) = 0$ for all $f,g \in D(Q)$. Moreover, $Q$ is called {\em monotone} if $|f| \leq |g|$ implies $Q(f) \leq Q(g)$ whenever $f,g \in D(Q)$. In this case, $Q(f)$ only depends on the absolute value of $f$ and not on its sign. We discuss these properties for forms whose domains are continuous functions in Appendix~\ref{section:bilinear_and_measure}.

\subsection{(Regular) Dirichlet forms}

A densely defined closed quadratic form $Q$ on $L^2(X,m)$ is called {\em Dirichlet form} if $f \in D(Q)$ implies $f_+ \wedge 1 \in D(Q)$ and $Q(f_+ \wedge 1) \leq Q(f)$. The {\em second Beurling-Deny criterion} \cite[Theorem XIII.51]{RS78} asserts that $Q$ is a Dirichlet form if and only if the semigroup $(e^{-tL})$ generated by the positive self-adjoint operator $L$ associated with $Q$ is {\em Markovian}, i.e., $0\leq f\leq 1$ implies $0\leq e^{-tL}f\leq 1$ for all $t\geq 0$.

If $Q$ is a Dirichlet form, then $D(Q) \cap L^\infty(X,m)$ is an algebra with respect to pointwise multiplication and 
$$Q(fg)^{1/2} \leq \norm{g}_\infty Q(f)^{1/2} + \norm{f}_\infty Q(g)^{1/2}  $$
for all $f,g \in D(Q) \cap L^\infty(X,m)$, see \cite[Theorem~1.4.2]{FOT}.

A Dirichlet form $Q$ is called {\em regular} if the following are satisfied:
\begin{itemize} 
 \item $X$ is a locally compact separable metric space and $m$ is a Radon measure of full support.
 \item $D(Q) \cap C_c(X)$ is uniformly dense in $C_c(X)$ and in $D(Q)$ with respect to $\norm{\cdot}_Q$.
\end{itemize}
In this case, the $Q$-{\em capacity} (or simply capacity if $Q$ is fixed) of an open set $O \subseteq X$ is defined by
$$\ca(O) = \inf\{\norm{f}_Q^2 \mid f \in D(Q) \text{ with } f \geq 1\, m\text{-a.e. on } O\}. $$
Here we use the convention $\ca(O)  = \infty$ if there does not exist $f \in D(Q)$ with $f \geq 1$ on $O$. For an arbitrary set $A \subseteq X$, the capacity is defined by
$$\ca(A) = \inf\{\ca(O) \mid O \text{ open with } A \subseteq O\}.$$
The capacity is {\em inner regular}, i.e., for any Borel set $A \subseteq X$ it satisfies
$$\ca(A) = \sup\{\ca(K) \mid K \text{ compact with } K \subseteq A\},$$
see \cite[Theorem~2.1.1]{FOT}. Moreover, by \cite[Lemma~2.2.7]{FOT}, the capacity for compact $K \subseteq X$ can alternatively be described as
$$\ca(K) = \inf\{\norm{f}_Q^2 \mid f \in D(Q) \cap C_c(X) \text{ with } f \geq 1\text{ on } K\}. $$
A subset $A$ of $X$ is called \emph{polar} if $\capa(A)=0$ holds. A property is said to hold \emph{quasi everywhere}, abbreviated q.e., if it holds on the complement of a polar set.

A measurable function $f:X\to [-\infty,\infty]$ is said to be \emph{quasi continuous} if for every $\varepsilon>0$ there is an open set $O$ with $\capa(O)<\varepsilon$ such that $f|_{X\setminus O}$ is finite-valued and continuous. If $Q$ is a regular Dirichlet form, then every $f$ in $D(Q)$ has a unique (up to equality quasi everywhere) quasi continuous representative $\tilde{f}$, cf.~\cite[Theorem 2.3.4]{ChenFuku}.

\subsection{Domination of Dirichlet forms and semigroups} \label{subsection:domination} 
If $U,V$ are sublattices of $L^2(X,m)$, we say that $U$ is an \emph{order ideal} in $V$ if $f \in U$, $g \in V$ and $|g| \leq |f|$ implies $g \in U$.

If $U,V$ are subalgebras of $L^\infty(X,m)$ we say $U$ is an \emph{algebraic ideal} in $V$ if $f \in U$ and $g \in V$ implies $fg \in U$. 

We will frequently use the following characterization. The equivalence of (i) and (ii) is Ouhabaz' domination criterion \cite[Theorem 3.7]{Ouh96}, whereas  the equivalence with (iii) is taken from \cite[Lemma~2.2]{S18}.

\begin{proposition}[Characterization of Domination]\label{prop:DirForm_domination}
Let $\q,\qp$ be Dirichlet forms with associated self-adjoint operators $L,L'$. The following assertions are equivalent.
\begin{enumerate}[(i)]
\item  For all nonnegative $f \in L^2(X,m)$ and all $t \geq 0$ we have
$$e^{-tL}f \leq e^{-tL'} f.$$
\item  $D(\q) \subseteq D(\qp)$, $D(\q)$ is an order ideal in $D(\qp)$ and 
       $$\q(f,g)\geq \qp(f,g)$$ 
       for all non-negative $f,g\in D(\q)$. 
 
\item   $D(\q) \subseteq D(\qp)$, $D(\q)\cap L^\infty(X,m)$ is an algebraic ideal in $D(\qp)\cap L^\infty(X,m)$  and
       $$\q(f,g)\geq \qp(f,g)$$ 
       for all non-negative $f,g\in D(\q)$. 
\end{enumerate}
\end{proposition}
If $Q$ and $Q'$ satisfy one of the conditions of this proposition, we say that $Q'$ {\em dominates} $Q$ and write $Q \preceq Q'$. Similarly, in this situation we write $(e^{-tL}) \preceq (e^{-tL'})$ and say that the semigroup  $(e^{-tL'})$ {\em dominates} the semigroup $(e^{-tL})$. 

Domination also induces an order relation on the set of all Dirichlet forms on $L^2(X,m)$.  Note that in general $\q \preceq \qp$ does not imply $\q \sqsubseteq \qp$ nor the other way round.

\section{The maximal dominating form and an abstract characterization of sandwiched semigroups}\label{sec:sandwiched_forms_abstract}

For every Dirichlet form $\q$ there is a maximal Dirichlet form $\QM$ (with respect to the natural order) that dominates the given Dirichlet form $\q$. In this section we describe the construction of this maximal form and give an abstract characterization of all Dirichlet forms $Q'$ that satisfy $\q \preceq  Q' \preceq \QM$ provided that $\QM$ is an extension of $\q$. 

We denote by $(T_t)$ the semigroup generated by $\q$ and by $(T_t^{(M)})$ the semigroup generated by $\QM$. According to the discussion in Subsection~\ref{subsection:domination}, any self-adjoint $C_0$-semigroup $(S_t)$ with 
$$(T_t) \preceq (S_t) \preceq  (T_t^{(M)}) $$
corresponds to a Dirichlet form $Q'$ with $\q \preceq  Q' \preceq \QM$. Hence, our result can be seen as an abstract characterization of all semigroups sandwiched between $(T_t)$ and $(T_t^{(M)})$. 

\subsection{The active main part and the killing part} We will next recall the definition of the active main part and the killing part of a Dirichlet form. For two concrete examples see Examples \ref{example:dirichlet integral}, \ref{example:fractional} below.

Let $\q$ be a Dirichlet form on $L^2(X,m)$. For $\varphi \in D(\q)$ with $0 \leq \varphi \leq 1$ we define the domain of the quadratic form $\tilde{\q}_\varphi$ on $L^2(X,m)$ by  
$$D(\tilde{\q}_\varphi) = \{f \in L^2(X,m) \cap L^\infty(X,m) \mid f\varphi, f^2 \varphi \in D(\q)\},$$
on which it acts by
$$\tilde{\q}_\varphi(f) = \q(\varphi  f) - \q(\varphi f^2,\varphi).$$
Since $D(Q) \cap L^\infty(X,m)$ is an algebra, we have $D(Q)\cap L^\infty(X,m) \subseteq D(\tilde{\q}_\varphi)$. The form $\tilde{\q}_\varphi$ is closable on $L^2(X,m)$. Indeed, \cite[Theorem~3.1]{S18} shows that $\tilde{\q}_\varphi$ is lower semicontinuous on its domain with respect to local convergence in measure and hence it is lower semicontinuous on its domain with respect to $L^2$-convergence. We denote its closure by $Q_\varphi$. The next proposition summarizes further important properties of $\q_\varphi$.

\begin{proposition}\label{proposition:properties qp}
Let $\varphi,\psi \in D(Q)$ with $0\leq \varphi \leq \psi \leq 1$.  
 \begin{enumerate}[(a)]
 \item $Q_\varphi$ is a Dirichlet form and its domain satisfies
 $$D(\q_\varphi) \cap L^\infty(X,m) = D(\tilde{\q}_\varphi) = \{f \in L^2(X,m) \cap L^\infty(X,m) \mid f \varphi \in D(\q)\}. $$
\item $D(\q) \subseteq D(\q_\varphi)$ and 
$$\q_\varphi(f) \leq \q(f), \quad f \in D(Q).$$
\item 
$D(Q_\psi) \subseteq D(Q_\varphi) $ and 
$$\q_\varphi(f) \leq \q_\psi(f), \quad f \in D(\q_\psi).$$

 \end{enumerate}
\end{proposition}
\begin{proof}
 This follows from  \cite[Theorem~3.18]{S18}. The proofs given there treat an extension of $Q_\varphi$ to all measurable $m$-a.e. defined functions that is lower semicontinuous with respect to local convergence in measure.   Restricting this form with larger domain to $L^2(X,m)$ yields all the claims.
 \end{proof}

\begin{remarks}
 Part (a) of this proposition is important because it yields a formula for $\q_\varphi$ for bounded functions in its domain. Namely, for $f,g \in  D(\q_\varphi) \cap L^\infty(X,m)$ we have $f,g \in D(\tilde{\q}_\varphi)$ and hence 
 $$\q_\varphi(f,g) = \tilde{\q}_\varphi(f,g)  = \q(\varphi f,\varphi g) - \q(\varphi fg,\varphi).$$
 For the last equality, we used the definition of $\tilde{\q}_\varphi$ and polarization.
\end{remarks}

\begin{definition}[Active main part]
The {\em active main part $\QM$} of $\q$ is defined as follows: Its domain $D(\QM)$ consists of all $f \in L^2(X,m)$ that satisfy $f \in D(\q_\varphi)$ for all $\varphi \in D(\q)$ with $0 \leq \varphi \leq 1$ such that 
$$\{\varphi \in D(\q) \mid 0 \leq \varphi \leq 1\} \to [0,\infty),\quad \varphi \mapsto \q_\varphi(f)$$
is bounded. On it $\QM$ acts by
$$\QM(f) = \sup\{\q_\varphi(f) \mid \varphi \in D(\q) \text{ with } 0\leq \varphi \leq 1\}.$$
\end{definition}
Since $\varphi \mapsto Q_\varphi(f)$ is monotone increasing, the form $\QM$ is indeed a Dirichlet form, see \cite[Theorem~3.6]{S18}.  It turns out that $\QM$ is the maximal Dirichlet form with respect to the natural order that dominates $\q$, i.e., $\q \preceq \QM$ and for all Dirichlet forms $\qp$ with $\q \preceq \qp$ we have $\qp \sqsubseteq \QM$, see \cite[Theorem 3.19]{S18}. However, $\QM$ need not be an extension of $Q$ and hence we introduce the following definition.
\begin{definition}[Killing part]
 The difference 
$$\QK = \q - \QM$$
with domain $D(\QK) = D(Q)$ is called the {\em killing part} of $Q$. 
\end{definition}
  The killing part is a local and positive quadratic form. Both properties are a consequence of $\QK$ being monotone, see \cite[Lemma~3.11]{S18} for monotonicity and \cite[Lemma~B.1]{S18} for how monotonicity implies the other properties. In particular, the value of $\QK(f)$ only depends on $|f|$ and not on the sign of $f$.


We illustrate these objects with an example. It shows that the active main part is an abstract way of constructing operators with Neumann boundary conditions from the quadratic forms leading to Dirichlet boundary conditions. 
 
\begin{example}[Dirichlet and Neumann Laplacian on domains]\label{example:dirichlet integral}
 Let $\Omega \subseteq \IR^n$ be open (or more generally let $\Omega$ be a Riemannian manifold) and let $V \in L^1_{\rm loc}(\Omega)$ be nonnegative.  We consider the Dirichlet form $\env$ with domain   $D(\env) = \{f \in H^1(\Omega) \mid V ^{1/2}f \in L^2(\Omega)\}$, on which it acts by
 $$\env(f) = \int_{\Omega} |\nabla f |^2 dx + \int_\Omega  |f|^2 V dx.$$
 The associated operator is the self-adjoint realization of the Schrödinger operator $H = -\Delta  + V$ with (abstract) Neumann boundary conditions, which we denote by $H^{(N)}$. Moreover, we let $\edv$ be the restriction of $\env$ to $D(\edv) = \{f \in H_0^1(\Omega) \mid V ^{1/2}f \in L^2(\Omega)\}$. This is a regular Dirichlet form and the associated operator is the self-adjoint realization of the Schrödinger operator $H = -\Delta  + V$ with (abstract) Dirichlet boundary conditions, which we denote by $H^{(D)}$.

The active main part of $\edv$ is given by  $\cE^{(N)}_0$.  Hence, the self-adjoint operator associated to the active main part is  $-\Delta^{(N)}$. For $f \in D(\edv)$, the killing part of $\edv$ is given by
 $$(\edv)^{(k)}(f) = \int_\Omega |f|^2 V dx.$$
In particular, if $V = 0$, this discussion shows that a Dirichlet form $Q$ satisfies $\cE^{(D)}_0 \preceq Q \preceq (\cE^{(D)}_0)^{(M)}$ if and only if the associated semigroup $(S_t)$ satisfies $(e^{t \Delta^{(D)}}) \preceq (S_t) \preceq   (e^{t \Delta^{(N)}})$. Hence, forms sandwiched between $\cE^{(D)}_0 $ and $(\cE^{(D)}_0)^{(M)} = \cE^{(N)}_0$ correspond to  semigroups sandwiched between the Dirichlet and the Neumann   semigroup of the Laplacian. This is precisely the situation studied in \cite{AW03}.

 \begin{proof}
   Here we only sketch the main ideas of the proof. For the details we refer to \cite[Example~3.9]{S18}.  We only consider bounded functions, the general case can be treated through approximations. 
   
   Let $f \in   H^1(\Omega) \cap L^\infty(\Omega)$  and let $\varphi \in C_c^\infty(\Omega)$ with $0 \leq \varphi \leq 1$. A direct computation using the product rule for $\nabla$ shows $f \in D((\edv)_\varphi)$ and 
 $$(\edv)_\varphi(f) = \int_{\Omega} \varphi^2 |\nabla f|^2 dx.$$
 Letting $\varphi \nearrow 1$ and taking into account that $C_c^\infty(\Omega)$ is dense in $D(\edv)$ yields $f \in D((\edv)^{(M)})$ and the formula for the action of $(\edv)^{(M)}$.
 
 Similarly, if $f \in D((\edv)^{(M)}) \cap L^\infty(\Omega)$,  by the definition of $(\edv)_\varphi$ and the active main part, we have $\varphi f \in D(\edv) =  H^1_0(\Omega) \cap L^2(\Omega, V \cdot dx)$ for every $\varphi \in C_c^\infty(\Omega)$. This yields $\nabla f \in \vec L^2_{\rm loc}(\Omega)$. With this at hand, an application of the product rule for $\nabla$ as above shows $(\edv)_\varphi(f) = \int_{\Omega} \varphi^2 |\nabla f|^2 dx.$ Since $(\edv)_\varphi(f) \leq (\edv)^{(M)}(f)$ and $\varphi$ is arbitrary, we conclude $\nabla f \in \vec L^2(\Omega)$ so that $f \in H^1(\Omega)$. 
 
 The statement on the killing part is an immediate consequence. 
 \end{proof}
\end{example}

\begin{example}[Fractional Laplacians] \label{example:fractional}
 As above we let $\Omega \subseteq \R^n$ be open. For a background on fractional Sobolev spaces we refer to \cite{NPV}. For $0 < s < 1$, we denote by $Q^{s,(N)}$ the Dirichlet form with domain $D(Q^{s,(N)}) = W^s(\Omega)$ on which it acts by
 $$Q^{s,(N)}(f) =  \frac{1}{2}\int_{\Omega \times \Omega} \frac{|f(x) - f(y)|^2}{|x-y|^{n + 2s}} dx\,dy.$$
 The restriction of this form to $W^s_0(\Omega)$ is denoted by $Q^{s,(D)}$, it is a regular Dirichlet form. Note that at least if $\Omega$ is bounded and has $C^\infty$-boundary, the spaces $W^s_0(\Omega)$ and $W^s(\Omega)$ coincide for $0<s\leq \frac 1 2$ by \cite[Theorem 11.1]{LM72}, which makes the problem of finding the Dirichlet forms sandwiched between $Q^{s,(D)}$ and $Q^{s,(N)}$ trivial. 
 
 It is well-known that the associated self-adjoint operators $H_s^{(N)}$ and  $H_s^{(D)}$ are restrictions of the \emph{restricted fractional Laplacian} $H_s$ given by
 $$ H_s f(x) =  {\rm P.V.} \int_{\Omega} \frac{f(x)-f(y)}{|x-y|^{n+2s}}dy = \lim_{\varepsilon \to 0+} \int_{\Omega\setminus B_\varepsilon(x)} \frac{f(x)-f(y)}{|x-y|^{n+2s}}dy .$$
 Hence, they can be viewed as realizations of $H_s$ with abstract Neumann and Dirichlet boundary conditions. Note that we ignore a constant so that our fractional Laplacian is only a constant multiple of the 'usual' restricted fractional Laplacian, cf. \cite[Section~3]{NPV}. Similar as in the previous example the active main part of $Q^{s,(D)}$ is $Q^{s,(N)}$.
 \begin{proof}
  Here we only show the statement on the active main part of $Q^{s,(D)}$, the rest is well-known. Since $Q^{s,(N)}$ and $(Q^{s,(D)})^{(M)}$ are Dirichlet forms, it suffices to prove $D(Q^{s,(N)}) \cap L^\infty(\Omega) = D((Q^{s,(D)})^{(M)}) \cap L^\infty(\Omega)$ and that $Q^{s,(N)}$ and $(Q^{s,(D)})^{(M)}$ agree on these sets (use that bounded functions are dense in the domains of Dirichlet forms, see \cite[Theorem~1.4.2]{FOT}). 
  
  We first proof that $Q^{s,(N)}$ is a restriction of $(Q^{s,(D)})^{(M)}$ (on $L^\infty(\Omega)$).   Let $f \in W^s(\Omega) \cap L^\infty(\Omega)$ and let $\varphi \in W_0^s(\Omega)$ with $0 \leq \varphi \leq 1$. Then $f \varphi \in W_0^s(\Omega)$. We infer  
  \begin{align*}
  Q^{s,(D)}(\varphi  f) - Q^{s,(D)}(\varphi f^2,\varphi) &= \frac{1}{2}\int_{\Omega \times \Omega} \frac{(\varphi(x)f(x) - \varphi(y)f(y))^2}{|x-y|^{n + 2s}} dx\,dy \\
  &\quad - \frac{1}{2}\int_{\Omega \times \Omega} \frac{(\varphi(x)f(x)^2 - \varphi(y)f(y)^2)(\varphi(x)-\varphi(y))}{|x-y|^{n + 2s}} dx\,dy\\
  &= \frac{1}{2}\int_{\Omega \times \Omega} \varphi(x)\varphi(y) \frac{|f(x) - f(y)|^2}{|x-y|^{n + 2s}} dx\,dy.
  \end{align*}
  Taking the supremum over such $\varphi$ yields $f \in  D((Q^{s,(D)})^{(M)})$ and $(Q^{s,(D)})^{(M)}(f) = Q^{s,(N)}(f)$. 
  
  It remains to prove  $D((Q^{s,(D)})^{(M)}) \cap L^\infty(\Omega) \subseteq W^s(\Omega)$. Let  $f \in D((Q^{s,(D)})^{(M)}) \cap L^\infty(\Omega)$. For $\varphi \in W_0^s(\Omega)$  with $0 \leq \varphi \leq 1$,  we have by definition of the main part $\varphi f, \varphi f^2  \in W_0^s(\Omega)$ and 
  \begin{align*}
   (Q^{s,(D)})^{(M)} (f) & \geq  Q^{s,(D)}(\varphi  f) - Q^{s,(D)}(\varphi f^2,\varphi)\\
   &= \frac{1}{2}\int_{\Omega \times \Omega} \varphi(x)\varphi(y) \frac{|f(x) - f(y)|^2}{|x-y|^{n + 2s}} dx\,dy.
  \end{align*}
  For the last equality we used the same computation as above. Since $\varphi$ was arbitrary, this shows $f \in W^s(\Omega)$.  
 \end{proof}

%
 
\end{example}

\begin{remarks}
 These examples show that it is a good intuition to think of a regular Dirichlet form $\q$ with $\QK = 0$ as being a form with `Dirichlet type' boundary conditions and $\QM$ being the 'same' form with `Neumann type' boundary conditions.  
\end{remarks}

\subsection{An abstract characterization of sandwiched semigroups and forms}

The main abstract result of this paper is the following characterization of Dirichlet forms sandwiched between a Dirichlet form without killing and its active main part.  

\begin{theorem}\label{theorem:abstract characterization}
 Let $\q,\qp$ be Dirichlet forms on $L^2(X,m)$ with $\QK = 0$. The following assertions are equivalent.
 \begin{enumerate}[(i)]
  \item $\q \preceq \qp \preceq \QM$.
  \item \begin{enumerate}[(a)]
          \item $D(Q') \subseteq D(Q^{(M)})$ and $D(Q')$ is an order ideal in $D(Q^{(M)})$.   
         \item $\qp - \QM$ is a positive and local form on $D(\qp)$.
         \item $Q'$ is an extension of $Q$.
        \end{enumerate}
 
 \end{enumerate}
%
%
\end{theorem}
\begin{proof}
(i)$\implies$(ii): (a) This is a consequence of Proposition~\ref{prop:DirForm_domination}.

 (b) The positivity of $\qp - \QM$  follows directly from  $\qp \preceq \QM$, cf. Proposition~\ref{prop:DirForm_domination}. In order to see that $\qp - \QM$ is local, we let $f,g \in D(\qp)$ with $fg = 0$. Without loss of generality we  may assume  $f,g \geq 0$, for otherwise we can decompose $f,g$ into positive and negative parts and use $f_\pm g_\pm = 0$. Since $\qp$ and $\QM$ are Dirichlet forms, we can further assume that $f,g$ are bounded.  As we already established positivity, it remains to prove $\qp(f,g) - \QM(f,g) \leq 0$.  

Let   $\varphi \in D(\q)$ with $0 \leq \varphi \leq 1$. According to Proposition~\ref{prop:DirForm_domination} we have $f\varphi, g\varphi \in D(Q)$, so that by Proposition~\ref{proposition:properties qp} $f,g \in D(Q_\varphi)$. Using $fg = 0$ and $Q \preceq \qp$ we obtain 
\begin{align*}
  \qp(f,g) - Q_\varphi(f,g)  &= \qp(f,g) - Q(\varphi f,\varphi g) + \q(\varphi fg,\varphi) \\
  &=\qp(f,g) - \qp(\varphi f,\varphi g)\\
  &= \qp((1-\varphi)f,g) + \qp(\varphi f, (1-\varphi)g).
\end{align*}
The functions $\eta=(1-\varphi)f$ and $\zeta=g$ are nonnegative and satisfy $\eta\zeta=0$.  The Dirichlet form property of $Q'$ implies
\begin{align*}
Q'(\eta+\zeta)=Q'(\abs{\eta+\zeta})=Q'(\abs{\eta-\zeta})\leq Q'(\eta-\zeta),
\end{align*}
from which we deduce $Q'(\eta,\zeta)\leq 0$ by bilinearity. The same argument applies to $\eta=\varphi f$ and $\zeta=(1-\varphi)g$ so that we obtain
\begin{equation*}
Q'(f,g)-Q_\varphi(f,g)=Q'((1-\varphi)f,g)+Q'(\varphi f,(1-\varphi)g)\leq 0.
\end{equation*}
By the definition of $\QM$ we can choose $\varphi$ such that $Q_\varphi(f,g)$ is arbitrarily close to $\QM(f,g)$ and hence  obtain locality.

(c) The domination $\q \preceq \qp \preceq \QM$ and $\QK = 0$ yield for all nonnegative $f,g \in D(Q)$ the inequality
$$\q(f,g) = \QM(f,g) \leq Q'(f,g) \leq Q(f,g).$$
By splitting functions into positive and negative parts this shows $Q = Q'$ on $D(Q)$.
 
%
%

(ii)$\implies$(i): $\qp \preceq \QM$ follows directly from (a) and (b) and the characterization of domination Proposition~\ref{prop:DirForm_domination}.

$\q \preceq \qp$: Since $D(\qp)$ is contained in $D(\QM)$ and $D(Q)$ is an order ideal in $D(\QM)$ we obtain that $D(\q)$ is also an order ideal in $D(\qp)$. Since $\qp$ is an extension of $\q$, this already implies domination. 
\end{proof}

We can rephrase this theorem slightly. Let $Q$ be a Dirichlet form with $\QK = 0$. We say that a pair $(F,q)$ consisting of a vector lattice $F \subseteq D(\QM)$ that is an order ideal in $D(\QM)$ and a quadratic form $q$ with $D(q) = F$ is an {\em abstract admissible pair} for $Q$, if it satisfies the following properties:
\begin{itemize}
 \item $D(Q) \subseteq D(q)$ and $q(f) = 0$ for $f \in D(Q)$,
 \item  $q$ is local and positive,
 \item the form $Q_{F,q}=\QM|_F+ q$ is closed. 
\end{itemize}

\begin{corollary}
 Let $\q$ be  Dirichlet forms with $\QK = 0$. The following assertions are equivalent.
 \begin{enumerate}[(i)]
  \item $\qp$ is a Dirichlet form with $\q \preceq \qp \preceq \QM$.
  \item There exists an abstract admissible pair $(F,q)$ such that $Q' = Q_{F,q}$.
 \end{enumerate}
\end{corollary}

\begin{proof}
 (i) $\Longrightarrow$ (ii): This is a reformulation of the previous theorem. 
 
 (ii) $\Longrightarrow$ (i): Using the previous theorem it suffices to show that $Q_{F,q}$ is a Dirichlet form. Since closedness and density of $D(Q_{F,q}) = F$ are part of the definition of abstract admissible pairs, it suffices to prove the Markov property.   By assumption $F$ is an order ideal in $D(Q^{(M)})$ and for $f \in F$ we have $f_+ \wedge 1 \in D(Q^{(M)})$ and  $|f_+ \wedge 1| \leq |f|$. This shows $f_+ \wedge 1 \in  F$ whenever $f \in F$. Moreover, as already discussed after introducing the killing part, $q$ being local and positive yields that $q$ is monotone, see \cite[Lemma~B.1]{S18}.  These observations and $Q^{(M)}$ being Markovian imply 
  \begin{equation*}
  Q_{F,q}(f_+ \wedge 1) = Q^{(M)}(f_+ \wedge 1) + q(f_+ \wedge 1) \leq Q^{(M)}(f) + q(f) = Q_{F,q}(f).\qedhere
  \end{equation*}
\end{proof}

\begin{remarks}
 This corollary shows that in order to determine all sandwiched forms between $Q$ and $\QM$ we need to characterize all abstract admissible pairs. This is possible when $\QM$ is a regular Dirichlet form on a metric space $\mathcal K$ containing $X$ as a dense open subset. In the next section we will prove that in this case:
 \begin{enumerate}[(a)]
  \item Positive and local forms correspond to measures if their domain contains sufficiently many continuous functions, see Appendix~\ref{section:bilinear_and_measure}. If these forms satisfy $q(f) = 0$ for $f \in D(Q)$, the corresponding measure is supported on the boundary $\mathcal K \setminus X$.
  \item Closed order ideals in $D(\QM)$  correspond to functions vanishing outside an  open set (under some additional density assumption for continuous functions). 
 \end{enumerate}
 This then allows us to identify abstract admissible pairs  with pairs of open subsets of the boundary and certain measures on them.  
\end{remarks}

\section{Domination for parts of regular Dirichlet forms} \label{section:domination regular}

%
%
%
%

Let $Q$ be a regular Dirichlet form on $L^2(X,m)$. Let $O \subseteq X$ be an open set and let $\mu$ be a Radon measure on the Borel $\sigma$-algebra of $O$.  We define the quadratic form $Q^c_{O,\mu}$ by letting $D(Q^c_{O,\mu}) = D(Q) \cap C_c(O)$ and 
$$Q^c_{O,\mu}(f) = Q(f) + \int_O f^2 d \mu.$$
Here $C_c(O)$ is tacitly identified with $\{ \varphi \in C_c(X) \mid {\rm supp}\, \varphi \subseteq O\}$.
\begin{proposition}\label{propostion:characterization admissibility}
 The following assertions are equivalent. 
 \begin{enumerate}[(i)]
  \item $\mu$ charges no sets of $Q$-capacity zero. 
  \item The quadratic form $Q^c_{O,\mu}$ is closable. 
 \end{enumerate}
In this case, the closure $Q_{O,\mu}$ of $Q^c_{O,\mu}$ is given by
\begin{align*}
D(Q_{O,\mu}) &= \{f \in D(Q) \mid \tilde f = 0 \text{ q.e. on } X \setminus O \text{ and } \int_O \tilde f^2 d\mu < \infty\},\\
Q_{O,\mu}(f) &= Q(f) + \int_O \tilde f^2 d\mu.
\end{align*}
\end{proposition}

%
%
%
 \begin{proof}
  (i)$\implies$(ii): This follows as in \cite[Theorem~1.2]{Sto92}.
  
  (ii)$\implies$(i): By the inner regularity of the capacity and the inner regularity of the Radon measure $\mu$ it suffices to show for compact sets $K \subseteq O$  that  ${\rm cap}(K) = 0$ implies $\mu(K) = 0$. 
  
  Let now $K \subseteq O$ be compact with ${\rm cap}(K) = 0$. Since $Q$ is regular, there exists a sequence $(\varphi_n)$ in $D(Q) \cap C_c(X)$ such that $\norm{\varphi_n}_Q \to 0$, $0 \leq \varphi_n \leq 1$ and $\varphi_n \geq1 $ on $K$.  Let $G$ be open and relatively compact with $K \subseteq G \subseteq O$. Using regularity of $Q$ again yields the existence of a function $\psi \in D(Q) \cap C_c(X)$ with $0 \leq \psi \leq 1$, $\psi = 1$ on $K$ and ${\rm supp}\, \psi \subseteq G$.
  
  We now consider $f_n := \psi \cdot \varphi_n$. Since $Q$ is a Dirichlet form, it satisfies $f_n \in D(Q)\cap C_c(O) = D(Q_{O,\mu}^c)$ and 
  $$Q(f_n)^{1/2} \leq Q(\psi)^{1/2} + Q(\varphi_n)^{1/2}.$$

  The inequality $0 \leq f_n \leq 1$ and ${\rm supp}\, f_n \subseteq G$ imply $\norm{f_n}  \to  0$ as $n\to \infty$ and 
  $$\int_{O} |f_n|^2 d\mu \leq \mu(G), \quad n\geq 1. $$
  In particular, these estimates show that $(f_n)$ is bounded with respect to the form norm  $\norm{\cdot}_{Q_{O,\mu}}$. Let $Q_{O,\mu}$ be the closure of $Q_{O,\mu}^c$, which exists by (ii). The Banach--Saks theorem implies that for some subsequence $(f_{n_k})$ the sequence of Césaro means
  $$g_N := \frac{1}{N}\sum_{k = 1}^N f_{n_k}$$
  converges to some $g \in D(Q_{O,\mu})$ with respect to $\norm{\cdot}_{Q_{O,\mu}}$. The form norm of $Q_{O,\mu}$ is larger than $\norm{\cdot}$ and hence we obtain  $g_n \to g$ with respect to $\norm{\cdot}$. But since $\norm{f_n} \to 0$, we conclude $g = 0$. By the choice of $(f_n)$ we also have $g_N \in D(Q) \cap C_c(O)$,  $0 \leq g_N \leq 1$ and $g_N \geq 1$ on $K$. Putting everything together we obtain
  \begin{align*}
   \mu(K) \leq \int_O |g_N|^2 d\mu \leq Q_{O,\mu}(g_N) \to 0 \text{ as } N\to \infty.
  \end{align*}
 This yields the desired $\mu(K) = 0$. 
 
The proof of the formula for $Q_{O,\mu}$ follows as in \cite[Proposition~1.1]{SV96}.
 \end{proof}

 \begin{definition}
  A pair $(O,\mu)$ satisfying one of the conditions of the previous theorem is called an {\em admissible pair} for the form $Q$. In this case, we write $Q_{O,\mu}$ for the closure of the form $Q_{O,\mu}^c$ above.
 \end{definition}

 \begin{proposition}\label{proposition:characterization domination}
  Let $(O_i,\mu_i)$, $i = 1,2$, be admissible pairs for $Q$. The following assertions are equivalent: 
  \begin{enumerate}[(i)]
   \item $Q_{O_1,\mu_1} \preceq Q_{O_2,\mu_2}$.
   \item ${\rm cap}(O_1 \setminus O_2) = 0$ and $\mu_2(A) \leq \mu_1(A)$ for every Borel set $A \subseteq O_1 \cap O_2$. 
  \end{enumerate}
   \end{proposition}
  \begin{proof}
  (ii)$\implies$(i): This follows  immediately from (ii) and the formula for $Q_{O_i,\mu_i}$ given in Proposition~\ref{propostion:characterization admissibility}.
  
  (i)$\implies$(ii):  Since $D(Q_{O_1,\mu_1})$ is a lattice and an order ideal in $D(Q_{O_2,\mu_2})$, we have $D(Q_{O_1,\mu_1}) \subseteq D(Q_{O_2,\mu_2})$. Hence, every $f \in D(Q_{O_1,\mu_1})$ satisfies $\tilde f = 0$ q.e. on $X \setminus O_2$. 
  
Now, suppose ${\rm cap}(O_1 \setminus O_2) > 0$. By \cite[Theorem~2.1.1]{FOT} there exists a compact set $K \subseteq O_1 \setminus O_2$ with ${\rm cap}(K) > 0$ and by \cite[Theorem~2.1.5]{FOT} there exists $f \in D(Q)$ with $0 \leq f \leq 1$ and $\tilde f = 1$ q.e. on $K$. By the regularity of $Q$ there exists $\varphi \in D(Q) \cap C_c(O_1)$ with $\varphi \geq 1$ on $K$. We obtain $\varphi f \in D(Q_{O_1,\mu_1})$ as $\varphi \tilde f = 0$ q.e. on $X \setminus O_1$ and 
  $$\int_{O_1} |\varphi \tilde f|^2 d\mu_1 \leq \int_{O_1} |\varphi|^2 d\mu_1 < \infty.  $$
  Furthermore, $\varphi \tilde f \geq 1$ q.e. on $K \subseteq X \setminus O_2$. This and ${\rm cap}(K) > 0$ are a contradiction to the fact that functions in  $D(Q_{O_1,\mu_1})$ vanish q.e. on $X \setminus O_2$.
 
  It remains to prove the inequality for the measures. Domination implies 
  $$Q(\varphi) + \int_{O_2} |\varphi|^2 d\mu_2 \leq  Q(\varphi) + \int_{O_1} |\varphi|^2 d\mu_1  $$
  for all nonnegative $\varphi \in  D(Q) \cap C_c(O_1)$.  For any compact set $K \subseteq O_1 \cap O_2$ and any open neighborhood $G$ of $K$ with $G \subseteq O_1 \cap O_2$ there exists $\psi \in C_c(X) \cap D(Q)$ with ${\rm supp} \, \psi \subseteq G$, $0\leq \psi \leq 1$ and $\psi\geq 1$ on $K$. Plugging this into the last inequality yields
  $$\mu_2(K) \leq \int_{O_2} |\psi|^2 d\mu_2 \leq \int_{O_1} |\psi|^2 d\mu_1 \leq \mu_1(G).$$
  Thus we obtain $\mu_2(K) \leq \mu_1(K)$ from the outer regularity of the Radon measure $\mu_1$. By inner regularity of Radon measures this implies the statement for all Borel sets. 
 \end{proof}
%
%

 \section{A boundary for regular Dirichlet forms and a characterization of sandwiched semigroups}\label{sec:boundary_reps}

 \subsection{A boundary for regular Dirichlet forms} Let $Q$ be a regular Dirichlet form on $L^2(X,m)$. In this subsection we introduce a locally compact separable metric space $\mathcal{K}$ that contains $X$ as an open subset and extend $m$ to a Radon measure $\hat m$ on $\mathcal K$ such that $\QM$ can be considered to be a regular form on $L^2(\mathcal{K},\hat m)$.

The spaces $C_c(X)$ and $L^2(X,m)$ are separable because $X$ is locally compact separable metric space and $m$ is a Radon measure. The map
$$L^2(X,m) \to D(\QM), \quad f \mapsto (L^{(M)} + 1)^{-1} f $$
is continuous with respect to the form norm $\norm{\cdot}_{\QM}$ (here $L^{(M)}$ denotes the positive self-adjoint operator associated with $Q^{(M)}$). It has dense image $D(L^{(M)})$ in $D(\QM)$ with respect to $\norm{\cdot}_{\QM}$, showing that $(D(\QM),\norm{\cdot}_{\QM})$ is also separable. Moreover, \cite[Theorem~4.3]{S18}  asserts that for a regular Dirichlet form $\q$ the space $D(\QM) \cap C_b(X)$ is dense in $D(\QM)$ with respect to $\norm{\cdot}_{\QM}$ (this is an abstract version of the Meyers-Serrin theorem). 

Combining these observations yields the existence of a subalgebra $\mathcal{C}$ of $D(\QM)\cap C_b(X)$ with the following three properties: 
\begin{itemize}
\item $\mathcal{C}$ is countably generated.
\item $\mathcal{C}$ is $\norm{\cdot}_{\QM}$-dense in $D(\QM)$.
\item $\mathcal{C}\cap C_c(X)$ is uniformly dense in $C_c(X)$.
\end{itemize}

Let $\mathcal A$ be the uniform closure of $\mathcal C$. Its  complexification $\mathcal{A}^\IC = \{f + ig \mid f,g \in \mathcal A\}$  is a  commutative $C^\ast$-algebra that satisfies $C_0(X;\IC) \subseteq \mathcal A^\IC \subseteq C_b(X;\IC)$. By Gelfand theory there exists a unique (up to homeomorphism) locally compact, separable Hausdorff space $\mathcal{K}$ with the following properties: 
\begin{itemize}
\item $X$ is a dense and open subset of $\mathcal K$.
\item Every $f \in \mathcal A^\IC$ can be   extended to a function $\hat f \in C_0(\mathcal{K};\IC)$ and 
$$C_0(\mathcal{K};\IC) = \{\hat f \mid f \in \mathcal A^\IC\}.$$
\end{itemize}
As $\mathcal C$ is countably generated, the space $\mathcal{K}$ is metrizable. Hence, $\mathcal{K}$ is Polish, that is, separable and completely metrizable, since every locally compact, separable, second countable space is Polish.  Since $X$ is dense in $\mathcal K$,  the continuous extension of a function from $\mathcal A$ to $\mathcal K$ is unique and we will therefore not distinguish between elements of $\mathcal A$ and their extension.  For real-valued functions this interpretation leads to  $\mathcal A = C_0(\mathcal K) (=C_0(\mathcal K;\IR))$.
%

The measure $m$ on $X$ can be extended to a Borel measure $\hat m$ on $\mathcal{K}$ by setting
\begin{align*}
\hat m(A)=m(A\cap X),\quad A\in \mathcal{B}(\K).
\end{align*}
The measure $\hat m$ is again a Radon measure of full support. By this definition the space $L^2(\mathcal{K},\hat m)$ can be naturally identified with $L^2(X,m)$ via the unitary map 
$$R \colon L^2(\mathcal{K},\hat m) \to L^2(X,m),\quad f \mapsto f|_{X}.$$
Our discussion shows $R^{-1} (\mathcal A \cap L^2(X,m))  = C_0(\mathcal K) \cap L^2(\mathcal K,\hat m)$. Since $R$ also preserves the order relation, any Dirichlet form  on $L^2(X,m)$ can be viewed as a Dirichlet form on $L^2(\mathcal{K},\hat m)$ under this transformation. In particular, the form $\QM$ is a regular Dirichlet form on $L^2(\mathcal{K},\hat m)$, see \cite[Theorem~4.4]{S18}.

The following remark sketches the uniqueness of the space $\mathcal K$. We leave details (especially the involved definitions, which can be found in \cite[Appendix~A.4]{FOT}) to the reader. 

\begin{remarks}[Uniqueness of $\mathcal K$] 
The space $\mathcal K$ depends on the choice of the algebra $\mathcal C$. However, given two  algebras $\mathcal C, \mathcal C'$ with the required properties and the corresponding spaces $\mathcal K, \mathcal K'$, there exists a unitary order isomorphism 
$$U \colon L^2(\mathcal K,\hat m) \to L^2(\mathcal K',\hat m')$$
such that $0\leq f_n\leq 1$, $f_n\nearrow 1$ implies $Uf_n\nearrow 1$ and $U$ intertwines $\QM$ and $\QM$ (when considererd as a form on the corresponding space). This implies that both forms are equivalent in the sense of \cite[Appendix~A.4]{FOT}. Since they are also regular, \cite[Theorem~A.4.2]{FOT} yields that $\mathcal K$ and $\mathcal K'$ are quasi-homeomorphic (and establishes further properties of a corresponding quasi-homeomorphism).  
\end{remarks}

In view of the previous remark we  make the following definition. 

\begin{definition}\label{def_boundary}
The set $\partial X=\mathcal{K}\setminus X$ is called the \emph{boundary} of $X$ relative to the form $Q$. 
\end{definition}

\begin{example}[Dirichlet and Neumann Laplacian -- continued] \label{example:boundary} We use the situation and notation of Example~\ref{example:dirichlet integral} and assume that the potential vanishes, i.e., $V = 0$. 

As discussed in Example~\ref{example:dirichlet integral} we have  $(\cE^{(D)}_0)^{(M)} = \cE^{(N)}_0$ so that  $D((\cE^{(D)}_0)^{(M)}) = H^1(\Omega)$ and the standard Sobolev norm on $H^1(\Omega)$ coincides with the form norm of $(\cE^{(D)}_0)^{(M)}$. If $\Omega \subseteq \R^n$ has continuous boundary (for a precise definition see e.g. \cite[Definition~4.1]{EE}), the space $\{f|_\Omega \mid f \in C_c^\infty(\R^n)\}$ is dense in $ H^1(\Omega)$ with respect to the standard Sobolev norm. Hence, in this case we can choose the algebra $\mathcal C$ to be a subset of $\{f|_{\Omega} \mid f \in C_c^\infty(\R^n)\} \subseteq C_c(\overline{\Omega})$. Since by Stone-Weierstra\ss\, $\{f|_\Omega \mid f \in C_c^\infty(\R^n)\}$ is dense in $C_0(\overline{\Omega})$, we can further assume that $\mathcal C$ is dense in  $C_0(\overline{\Omega})$, so that the algebra $\mathcal A$, the uniform closure of $\mathcal C$, equals $C_0(\overline{\Omega})$. Our construction then yields $\mathcal K =  \overline{\Omega}$ (up to homeomorphism) and that the boundary of $\Omega$ relative to $\cE_0^{(D)}$ coincides with the metric boundary $\partial \Omega = \overline{\Omega} \setminus \Omega$ in $\R^n$.
\end{example}
\begin{example}[Fractional Laplacian -- continued] \label{example:boundary fractional}
We use the situation and notation of Example~\ref{example:fractional}. As discussed above we have $D((Q^{s,(D)})^{(M)}) = W^s(\Omega)$.   Moreover, if $\Omega$ has Lipschitz boundary, then $\{f|_{\Omega} \mid f \in C_c^\infty(\R^n)\}$ is dense in $W^s(\Omega)$ with respect to the form norm of $Q^{s,(N)}$, which coincides with the ususal norm on $W^s(\Omega)$, see \cite[Corollary~5.5]{NPV}. With this at hand the same argument as in  the previous example yields that we can choose $\mathcal K = \overline{\Omega}$ such that the boundary of $\Omega$ relative to $Q^{s,(D)}$ coincides with the metric boundary $\partial \Omega = \overline{\Omega} \setminus \Omega$ in $\R^n$.
\end{example}

 \subsection{A characterization of sandwiched semigroups for regular Dirichlet forms}
 
 In this subsection we prove the main result of this paper. Let $Q$ be a regular Dirichlet form on $L^2(X,m)$. We apply the theory developed in Section~\ref{section:domination regular} to the form $\QM$ when considered as a regular Dirichlet form on $L^2(\mathcal K,\hat m)$. We start with a simple observation that follows from the previous discussion.
 
 \begin{proposition}\label{proposition:identities}
  Let $Q$ be a regular Dirichlet form. Then $(\mathcal K,0)$ and $(X,0)$ are admissible pairs for the regular Dirichlet form $\QM$ on $L^2(\mathcal K,\hat m)$ and we have  $\QM = (\QM)_{\mathcal K,\,0}$ and $Q = (\QM)_{X,\,0}$. In particular,  
  $$D(Q) = \{f \in D(\QM) \mid \tilde f = 0 \text{ q.e. on } \partial X\}.$$
 \end{proposition}

 The following is the main result of the paper. 
 
 \begin{theorem}\label{theorem:main}
  Let $\q$ be a regular Dirichlet form with $\QK = 0$. For a Dirichlet form $Q'$ the following assertions are equivalent:
  \begin{enumerate}[(i)]
   \item There exists an admissible pair $(O,\mu)$ for $\QM$ with $X \subseteq O$ and $\mu(X) = 0$ such that $Q' = (Q^{(M)})_{O,\mu}.$  
   \item \begin{enumerate}[(a)]
          \item $\q \preceq \qp \preceq \QM$
          \item $D(Q') \cap C_c(\mathcal K)$ is dense in $D(Q')$ with respect to $\norm{\cdot}_{Q'}$.
         \end{enumerate}
  \end{enumerate}
 \end{theorem}
\begin{proof}

 (i)$\implies$(ii): (a) follows from Proposition~\ref{proposition:characterization domination} and the identities discussed in Propostion~\ref{proposition:identities}. The density of $ D((Q^{(M)})_{O,\mu}) \cap C_c(O) $ in $D((Q^{(M)})_{O,\mu})$ with respect to $\norm{\cdot}_{(Q^{(M)})_{O,\mu}}$ is part of the definition of the form $(Q^{(M)})_{O,\mu}$.
 
 (ii)$\implies$(i): Let $\mathcal D$ be the uniform closure of the algebra $D(Q') \cap C_c(\mathcal K)$. Since $D(Q') \cap C_c(\mathcal K)$ is an algebraic ideal in $D(\QM) \cap C_c(\mathcal K)$ (here we use domination and Proposition~\ref{proposition:characterization domination}), $\mathcal D$ is a uniformly closed ideal in
 $$\overline{D(\QM) \cap C_c(\mathcal K)}^{\norm{\cdot}_\infty} = C_0(\mathcal K)$$
 (here we use the regularity of $\QM$). Moreover, by Theorem~\ref{theorem:abstract characterization} we have $D(Q) \subseteq D(Q')$ so that $D(Q) \cap C_c(X) \subseteq D(Q') \cap C_c(\mathcal K)$. Since $Q$ is regular on $L^2(X,m)$, this yields $C_0(X) \subseteq \mathcal D$. By the characterization of closed ideals in $C_0(\mathcal K)$ there exists an open set $X \subseteq O \subseteq \mathcal K$ such that 
 $$\mathcal D = \{f \in C_0(\mathcal K) \mid f = 0 \text{ on } \mathcal K \setminus O\}.$$
Altogether this discussion shows that $D(Q') \cap C_c(O)$ is $\norm{\cdot}_{Q'}$ dense in $D(Q')$ and uniformly dense in $C_c(O)$.

Next, we show $D(Q') \cap C_c(O) = D(\QM) \cap C_c(O)$. Let $\varphi \in D(\QM) \cap C_c(O)$ and let $K = {\rm supp}\, \varphi \subseteq O$. Since $Q'$ is a Dirichlet form and $D(Q') \cap C_c(O)$ is uniformly dense in $C_c(O)$, there exists $\psi \in D(Q') \cap C_c(O)$ with $\psi = 1$ on $K$. We obtain $\varphi = \psi \varphi \in D(Q')\cap C_c(O)$ since $D(Q')\cap C_c(O)$ is an algebraic ideal in $D(\QM) \cap C_c(O)$ (here we use domination and Proposition~\ref{proposition:characterization domination}).

According to Theorem~\ref{theorem:abstract characterization} the domination (a) implies that the form $q = Q' - \QM$ with domain $D(q) = D(\QM) \cap C_c(O)$ is positive, local and satisfies $q(f) = 0$ for all $f \in D(Q) \cap C_c(X)$. By Corollary~\ref{coro:representation theorem} there exists a Radon measure $\mu$ on $O$ such that 
$$q(f) = \int_{O} |f|^2 d\mu, \quad f \in D(\QM) \cap C_c(O).$$
Since $D(Q) \cap C_c(X)$ is uniformly dense in $C_c(X)$, the property $q(f) = 0$ for $f \in D(Q)\cap C_c(X)$ implies $\mu(X) = 0$.

For $f \in D(\QM) \cap C_c(O)$,  we have by definition of $q$ that
$$Q'(f) = \QM(f) + \int_O |f|^2 d\mu = (\QM)^c_{O,\mu}(f).  $$
Since $Q'$ is closed and $D(\QM) \cap C_c(O)$ is $\norm{\cdot}_{Q'}$-dense in $D(Q')$, this implies that $(O,\mu)$ is an admissible pair for $\QM$ and $Q' = (\QM)_{O,\mu}$. 
%
%
%
%
\end{proof}

We can reformulate this theorem as follows. 

\begin{corollary}
  Let $\q$ be a regular Dirichlet form with $\QK = 0$. For a Dirichlet form $Q'$, the following assertions are equivalent. 
  \begin{enumerate}[(i)]
   \item There exists an open subset $\partial_\mu X \subseteq \partial X$ and a Radon measure $\mu$ on $\partial_\mu X$ that does not charge sets of $\QM$-capacity zero such that 
   $$D(Q') = \{f \in D(\QM) \mid \tilde f = 0 \text{ q.e. on } \partial X \setminus \partial_\mu X \text{ and } \int_{\partial_\mu X} |\tilde f|^2 d\mu < \infty\}$$
   and 
   $$Q'(f) = \QM(f) + \int_{\partial_\mu X} |\tilde f|^2 d\mu.$$
   \item \begin{enumerate}[(a)]
          \item $\q \preceq \qp \preceq \QM$
          \item $D(Q') \cap C_c(\mathcal K)$ is dense in $D(Q')$ with respect to $\norm{\cdot}_{Q'}$.
         \end{enumerate}
  \end{enumerate}
  
\end{corollary}

As an application of this result and our examples we obtain one of the main results of \cite{AW03} under slightly less restrictive assumptions.

\begin{example}\label{exa:arendt warma}
 Again we use the situation of Schrödinger operators on $\Omega$ of Example~\ref{example:dirichlet integral} with $V = 0$.  Assume further that $\Omega \subseteq \R^n$ has  continuous boundary and let $Q$ be a Dirichlet form on $L^2(\Omega)$ with associated Markovian semigroup $(S_t)$.   Let  $\partial \Omega = \overline{\Omega} \setminus \Omega$ be the metric boundary of $\Omega$. The disscusion in Example~\ref{example:dirichlet integral} and Example~\ref{example:boundary} combined with the previous corollary yield that the following assertions are equivalent.
 \begin{enumerate}[(i)]
 \item There exists an open subset $\partial_\mu \Omega \subseteq \partial \Omega$ and a Radon measure $\mu$ on $\partial_\mu \Omega$ that does not charge sets of $\mathcal E_0^{(N)}$-capacity zero such that 
   $$D(Q) = \{f \in H^1(\Omega) \mid \tilde f = 0 \text{ q.e. on } \partial \Omega \setminus\partial_\mu \Omega \text{ and } \int_{\partial_\mu \Omega} |\tilde f|^2 d\mu < \infty\}$$
   and 
   $$Q(f) = \int_\Omega |\nabla f|^2 dx + \int_{\partial_\mu \Omega} |\tilde f|^2 d\mu.$$
   \item \begin{enumerate}[(a)]
          \item $(e^{t\Delta^{(D)}}) \preceq (S_t) \preceq (e^{t\Delta^{(N)}})$
          \item $D(Q) \cap C_c(\overline{\Omega})$ is dense in $D(Q)$ with respect to $\norm{\cdot}_{Q}$. \end{enumerate}
 \end{enumerate}
This is precisely the statement of \cite[Theorem~4.1]{AW03} under the slightly less restrictive assumption of $\Omega$ having continuous boundary instead of Lipschitz boundary. 
\end{example}

\begin{example}\label{exa:arendt warma - fractional}
 We use the situation of fractional Laplacians of Example~\ref{example:fractional}.  Assume further that $\Omega \subseteq \R^n$ has Lipschitz boundary and let $Q$ be a Dirichlet form on $L^2(\Omega)$ with associated Markovian semigroup $(S_t)$.   Let  $\partial \Omega = \overline{\Omega} \setminus \Omega$ be the metric boundary of $\Omega$. The disscusion in Example~\ref{example:fractional} and Example~\ref{example:boundary fractional} combined with the previous corollary yield that the following assertions are equivalent.
 \begin{enumerate}[(i)]
 \item There exists an open subset $\partial_\mu \Omega \subseteq \partial \Omega$ and a Radon measure $\mu$ on $\partial_\mu \Omega$ that does not charge sets of $Q^{s,(N)}$-capacity zero such that 
   $$D(Q) = \{f \in W^s(\Omega) \mid \tilde f = 0 \text{ q.e. on } \partial \Omega \setminus\partial_\mu \Omega \text{ and } \int_{\partial_\mu \Omega} |\tilde f|^2 d\mu < \infty\}$$
   and 
   $$Q(f) = \frac{1}{2} \int_{\Omega \times \Omega} \frac{|f(x) - f(y)|^2}{|x-y|^{n + 2s}} dx\,dy + \int_{\partial_\mu \Omega} |\tilde f|^2 d\mu.$$
   \item \begin{enumerate}[(a)]
          \item $(e^{- t H_s^{(D)}}) \preceq (S_t) \preceq (e^{- t H_s^{(N)}})$
          \item $D(Q) \cap C_c(\overline{\Omega})$ is dense in $D(Q)$ with respect to $\norm{\cdot}_{Q}$. \end{enumerate}
 \end{enumerate}
The implication (i)$\implies$(ii) was also proved for Dirichlet forms associated with a related, but different fractional Laplacian by Claus and Warma \cite[Theorem~4.2]{CW20}.
\end{example}

As mentioned in the introducion we wanted to provide a version of the results of \cite{AW03} for general Dirichlet forms. In the abstract framework we were as general as possible but held back in generality for regular Dirichlet forms. In the following remarks we collect what else can be deduced from our general framework (at the cost of brevity and technical simplicity). 

\begin{remarks}
 \begin{enumerate}[(a)]
  \item  Another result of \cite{AW03} is the descritption of the operators corresponding to semigroups  $(e^{t\Delta^{(D)}}) \preceq (S_t) \preceq (e^{t\Delta^{(N)}})$ as Laplacians with Robin type boundary conditions. Something similar is possible here after equipping the abstract boundary $\partial X$ with so-called harmonic measures. This allows for the definition of densities of normal derivatives and leads to abstract Robin boundary conditions. In the Euclidean setting with $\Omega$ having Lipschitz boundary the harmonic measures  are mutually absolutely continuous with respect to the  surface measure on $\partial \Omega$ and the abstract normal derivatives are given by the usual normal derivative.   
  
  \item In Theorem~\ref{theorem:main} we  used that $D(Q') \cap C_c(\mathcal K)$ is dense in $D(Q')$ because we   constructed the set $O$ as the complement of the zero set of the closed ideal $$\overline{D(Q') \cap C_c(\mathcal K)}^{\norm{\cdot}_\infty}$$  in $C_0(\mathcal K)$. One can drop the density assumption and replace this argument by  the characterization of closed ideals in regular Dirichlet spaces given in \cite{St93}. In this case, Theorem~\ref{theorem:main} remains true without assertion (ii)(b) but with $O$ open  replaced by $O$ quasi-open. 
  \item We always assumed that the killing part vanishes. If $\QK\neq 0$, then there are two possible choices of reference for the maximal form:
  \begin{enumerate}[(1)]
   \item One can characterize all Dirichlet forms $Q'$ with $Q \preceq Q' \preceq \QM$ via abstract admissible pairs. Since in this case $\QM$ is not an extension of $Q$, the form  $q$ of the abstract admissible pair corresponding to $Q'$ does not vanish on $D(Q)$ but is bounded above by $\QK$. In the regular setting this implies that the measure $\mu$ from the admissible pair corresponding to $Q'$ is not necessarily supported only on $\partial X$.  It  satisfies $\mu \leq k$ on $X$, where $k$ is the measure corresponding to the local and positive form $\QK$ (cf. Appendix~\ref{section:bilinear_and_measure}). 
   \item Instead of comparing $Q'$ with $\QM$ one can characterize  $Q \preceq Q' \preceq Q^{\rm ref}$, where $Q^{\rm ref}$ is the active reflected Dirichlet form of $Q$. It arises by adding a suitable extension of $\QK$ to $\QM$, cf. \cite[Section~3.3]{S18}. In this case, our main theorems  still hold true but the proofs become substantially longer. 
  \end{enumerate}
 \end{enumerate}

\end{remarks}

\appendix
\section{Bilinear forms on $C_c(X)$}\label{section:bilinear_and_measure}

Let $X$ be a locally compact metric space. In this section we provide a characterization of positive and local forms defined on $C_c(X)$. First we show that densely defined positive forms on $C_c(X)$ can be extended to the whole of $C_c(X)$ if their domain is a lattice. In a second step we prove a representation theorem. Certainly both results are well-known to experts. Since we could not find a proper reference, we include the proofs for the convenience of the reader.  

%

In the following lemma we write $C(K)$ for the subspace $\{f \in C_c(X) \mid \supp f\subseteq K\}$.

\begin{lemma}\label{lemma:cont_ext_induct}
Let $q$ be a  densely  defined (with respect to the uniform norm) quadratic form on $D(q)\subseteq C_c(X)$. Suppose $q$ is positive and $D(q)$ is a lattice. 
\begin{enumerate}[(a)]
 \item For any compact $K \subseteq X$ the restriction of $q$ to $D(q) \cap C(K)$ is continuous. 
 \item $q$ can be uniquely extended to a positive quadratic form on $C_c(X)$.
\end{enumerate}
\end{lemma}
\begin{proof} We first show that for any compact set $K \subseteq X$ the restriction of $q$ to $D(q) \cap C(K)$ is continuous with respect to the supremum norm.

 Let $f,g\in D(q)\cap C(K)$ be nonnegative. Let $\theta_K\in D(q)$ be such that $ \theta_K\geq 0$ and $\theta_K\geq1$ on $K$. Such a functions exists because $D(q)$ is a dense lattice  in $C_c(X)$. Without loss of generality we  assume 
 $$q(\norm{f}_{\infty}\theta_K,g)-q(\norm{g}_{\infty}\theta_K,f)\leq 0,$$
 for otherwise we could interchange $f$ and $g$. Then, using the positivity of $q$, we get
\begin{align*}
0 &\leq  q(\norm{f}_{\infty}\theta_K-f,\norm{g}_{\infty}\theta_K+g)\\
&=-q(f,g)+\norm{f}_{\infty}\norm{g}_{\infty} q(\theta_K,\theta_K) +q(\norm{f}_\infty \theta_K,g)-q (\norm{g}_{\infty}\theta_K,f)\\
&\leq  -q(f,g)+\norm{f}_{\infty}\norm{g}_{\infty} q(\theta_K,\theta_K).
\end{align*}
This implies
$$0\leq q(f,g)\leq \norm{f}_\infty \norm{g} q(\theta_K,\theta_K).$$
For arbitrary  $f,g\in D(q)\cap C(K)$ we have $f^+,f^-,g^+, g^-\in D(q)\cap C(K)$ because $D(q)$ is a lattice. We obtain 
\begin{align*}
 |q(f,g)| &\leq q(f^+,g^+)+q(f^+,g^-)+q(f^-,g^+)+q(f^-,g^-)\\ 
  &\leq 4\norm{f}_\infty\norm{g}_\infty q(\theta_K,\theta_K).
\end{align*}

Using this continuity in order to prove that $q$ can be uniquely extended to a positive quadratic form on $C_c(X)$ it suffices to show the following:  For every nonnegative $\varphi \in C_c(X)$ there exists a compact $K\subseteq X$ with $\supp \varphi \subseteq K$ such that $\varphi$ can be approximated by nonnegative functions in $D(q) \cap C(K)$. 

To this end, we choose a nonnegative $\theta \in D(q)$ with $\theta \geq \norm{\varphi}_\infty$ on $\supp \varphi$. Such a function exists because $D(q)$ is a dense lattice. Let $K= \supp \theta.$ Since $q$ is densely defined, there exists $(\tilde \varphi_n)$ in $D(q)$ with $\tilde \varphi_n \to \varphi$ uniformly, as $n \to \infty$. Since $D(q)$ is a lattice, the sequence
$$\varphi_n = (\tilde \varphi_n)^+ \wedge \theta $$
belongs to $D(q)$. It is nonnegative and $\supp \varphi_n \subseteq \supp \theta = K$ for all $n \geq 1$. Moreover, using that  $0 \leq \varphi \leq \theta$, we obtain $\varphi_n \to \varphi$ uniformly, as $n \to \infty$.   
\end{proof}
The following theorem provides a characterization of monotone quadratic forms on $C_c(X)$. 
%
%
%
%
%
\begin{theorem}
 Let $q \colon C_c(X) \to [0,\infty)$ be a quadratic form. The following assertions are equivalent: 
\begin{enumerate}[(i)]
 \item $q$ is positive and local.
 \item For all $f,g \in C_c(X)$ the inequality $fg \geq 0$ implies $q(f,g) \geq 0$.
 \item For all $f,f',g,g' \in C_c(X)$ the inequality $fg \geq f'g'$ implies $q(f,g) \geq q(f',g')$.
 \item $q$ is monotone.
 \item There exists a Radon measure $\mu$ on $X$ such that 
 $$q(u) = \int_X f^2 d\mu, \quad f \in C_c(X).$$
\end{enumerate}
In this case, the measure $\mu$ is unique.
\end{theorem}
 
\begin{proof}
 Clearly, (ii) implies (i), (iii) implies (ii) and (v) implies all other assertions.
 
 (i)$\implies$(iv): Let $f,g\in C_c(X)$ with $|g|\leq |f|$.  The positivity of $q$ yields
\begin{align*}
q(|f|)= q(|g|,|f|)+q(|f|-|g|,|f|)\geq q(|g|,|f|)= q(|g|)+q(|g|,|f|-|g|) \geq q(|g|).
\end{align*}
It is left to show $q(f)=q(|f|)$ for every $f\in C_c(X)$. Since $f^+,f^-\in C_c(X)$ and $f^+ f^-=0$, the locality of $q$ implies $q(f^+,f^-)=0$ and hence 
\begin{equation*}
q(f)=q(f^+)-2q(f^+,f^-)+q(f^-)=q(f^+)+2q(f^+,f^-)+q(f^-)=q(|f|).
\end{equation*}

(iv)$\implies$(ii): Let $f,g \in C_c(X)$ with $fg \geq 0$. Then $|f + g| \geq |f - g|$ so that by monotonicty 
$$q(f) + q(g) - 2q(f,g)  = q(f-g) \leq q(f + g) \leq q(f) + q(g) + 2q(f,g).$$
This shows (ii).

We already proved the equivalence of (i),(ii) and (iv) and that these assertions are implied by (iii). Next we prove that they imply (iii).

Let $f,f',g,g' \in C_c(X)$ with $fg \geq f'g'$. If $f = f'$, the inequality $q(f,g) \geq q(f',g')$ directly follows from (ii). With the help of an approximation we reduce the case $f \neq f'$ to this one. 

We start with the following observation: Locality of $q$ implies that for $\varphi,\chi \in C_c(X)$ the value $q(\varphi,\chi)$ is independent of $\chi$ as long as $\chi = 1$ on ${\rm supp}\, \varphi$. In this case, we write $I(\varphi) := q(\varphi,\chi)$.

Let $\varepsilon > 0$. By compacteness of the supports  we can choose finitely many relatively compact open sets $G_j, j = 1,\ldots,N$, that cover the union of the supports of $f,f',g,g'$,  and choose $\xi_j \in G_j$, such that
$$\sup_{x \in G_j} |f(x) - f(\xi_j)| < \varepsilon \text{ and } \sup_{x \in G_j} |f'(x) - f'(\xi_j)| < \varepsilon.$$
We let $\chi_j \in C_c(X)$, $j = 1,\ldots,N$, be a subordinate partition of unity, i.e. $0 \leq \chi_j \leq 1,$ ${\rm supp} \chi_j \subseteq G_j$ and
$$\sum_{j = 1}^N \chi_j = 1 \text{ on } \bigcup_{j  =1}^N G_j.$$
Such a partition of unity exists because metric spaces are normal. We define 
$$\tilde{f} = \sum_{j = 1}^N f(\xi_j) \chi_j \text{ and } \tilde{f}' = \sum_{j = 1}^N f'(\xi_j) \chi_j.$$
Then, using $\sum_j \chi_j =1$ on the supports of $f,g$, we obtain
$$|q(\tilde{f},g) - q(f,g)|  \leq \sum_{j = 1}^N |q(\chi_j (f - f(\xi_j)),g)| \leq \varepsilon q(\sum_{j = 1}^N \chi_j, |g|) = \varepsilon I(|g|). $$
For the second inequality we used $|q(\varphi,\psi)| \leq q(|\varphi|,|\psi|)$, which directly follows from the positivity of $q$, and the fact that $|\chi_j (f - f(\xi_j))| \leq \varepsilon \chi_j$. Similarly, we have $|q(\tilde{f}',g') - q(f',g')| \leq \varepsilon I(|g'|)$. Moreover, 
\begin{align*}
 q(\tilde f,g) - q(\tilde f',g') &= \sum_{j = 1}^N (q(f(\xi_j) \chi_j,g) - q(f'(\xi_j) \chi_j,g'))\\
 &= \sum_{j = 1}^N q( \chi_j,f(\xi_j)g -f'(\xi_j)g')\\
 &\geq  \sum_{j = 1}^N q( \chi_j,f g -f' g') - \varepsilon q(\sum_{j = 1}^N \chi_j, |g| + |g|')\\
 &\geq - \varepsilon I(|g| + |g'|).
 \end{align*}
For the first inequality we used (ii) and the estimate
$$\chi_j(f(\xi_j)g -f'(\xi_j)g') \geq \chi_j (fg - f'g' -  \varepsilon (|g| + |g'|)).$$
The last inequality follows from $\chi_j(fg  - f'g') \geq 0$ and $\sum_j \chi_j = 1$ on the support of $|g| + |g'|$. Since $\varepsilon > 0$ was arbitrary, these estimates show (iii).

(iii)$\implies$(v): As above we define $I \colon C_c(X) \to \R$ by letting 
$$I(\varphi) = q(\chi,\varphi)$$
for some $\chi \in C_c(X)$ with $\chi = 1$ on the support of $\varphi$. It follows from (iii) that this is well-defined and positive. Moreover, $I$ is linear. By the Riesz-Markov-Kakutani representation theorem there exists a unique Radon measure $\mu$ such that 
$$I(\varphi)  =  \int_X \varphi d\mu$$
for all $\varphi \in C_c(X)$. Let now $f,g \in C_c(X)$ an let $\chi \in C_c(X)$ such that $\chi = 1$ on the supports of $f$ and $g$. Since   $fg = \chi (fg)$, property (iii) yields
$$q(f,g)= q(\chi,fg) = I(fg) = \int_X fg d\mu. $$
Thus, $\mu$ is the desired measure. 
\end{proof}

\begin{remarks}
 The statement of the theorem is not only valid for quadratic forms on continuous functions. The equivalence of (ii) and (iv) was observed in \cite[Appendix~B]{S18} for quadratic forms on sublattices of $L^0(Y,m)$, where $Y$ is an arbitrary set and $m$ is a measure on $Y$. Indeed, the above proof yields the equivalence of (i),(ii) and (iv) in this situation. The equivalence with (iii) requires the existence of suitable partitions of unity in the domain of $q$ and the equivalence with (v) requires that the domain of $q$ is an algebra and   a representation theorem for positive functionals.
\end{remarks}

\begin{corollary}\label{coro:representation theorem}
 Let $q$ be a densely defined positive and local quadratic form on $C_c(X)$ such that $D(q)$ is a lattice. Then there exists a unique Radon measure $\mu$ on $X$ such that 
 $$q(f) = \int_X f^2 d\mu, \quad f \in D(q).$$
\end{corollary}
\begin{proof}
As noted in the previous remark the form $q$ is also monotone. By Lemma~\ref{lemma:cont_ext_induct} it can be uniquely extended to a positive quadratic form on $C_c(X)$ and by the continuity of restrictions to compact sets this extension is also monotone. Hence, the statement follows from the previous theorem.  
\end{proof}

\bibliography{boundarybib}{}

\begin{thebibliography}{{Sch}20b}

\bibitem[ACD21]{ACD21}
Sahiba Arora, Ralph Chill, and Jean-Daniel Djida.
\newblock Domination of semigroups generated by regular forms.
\newblock 2021.

\bibitem[Akh18]{AK18}
Khalid Akhlil.
\newblock Locality and domination of semigroups.
\newblock {\em Results Math.}, 73(2):Art. 59, 11, 2018.

\bibitem[AW03]{AW03}
Wolfgang Arendt and Mahamadi Warma.
\newblock Dirichlet and {N}eumann boundary conditions: {W}hat is in between?
\newblock {\em J. Evol. Equ.}, 3(1):119--135, 2003.
\newblock Dedicated to Philippe B\'{e}nilan.

\bibitem[CF12]{ChenFuku}
Zhen-Qing Chen and Masatoshi Fukushima.
\newblock {\em Symmetric {M}arkov processes, time change, and boundary theory},
  volume~35 of {\em London Mathematical Society Monographs Series}.
\newblock Princeton University Press, Princeton, NJ, 2012.

\bibitem[Cla21]{Cla21}
Burkhard Claus.
\newblock {\em Non-linear Dirichlet forms}.
\newblock PhD thesis, TU Dresden, Dresden, 2021.

\bibitem[CW12]{CW12}
Ralph Chill and Mahamadi Warma.
\newblock {Dirichlet and Neumann boundary conditions for the $p$-Laplace
  operator: what is in between?}
\newblock {\em Proceedings of the Royal Society of Ediburgh: Section A
  Mathematics}, 142(5):975--1002, 2012.

\bibitem[CW20]{CW20}
Burkhard Claus and Mahamadi Warma.
\newblock {Realization of the fractional Laplacian with nonlocal exterior
  conditions via form methods}.
\newblock {\em J. Evol. Equ.}, 20(4):1597--1631, 2020.

\bibitem[DNPV12]{NPV}
Eleonora Di~Nezza, Giampiero Palatucci, and Enrico Valdinoci.
\newblock Hitchhiker's guide to the fractional {S}obolev spaces.
\newblock {\em Bull. Sci. Math.}, 136(5):521--573, 2012.

\bibitem[EE18]{EE}
D.~E. Edmunds and W.~D. Evans.
\newblock {\em Spectral theory and differential operators}.
\newblock Oxford Mathematical Monographs. Oxford University Press, Oxford,
  2018.

\bibitem[FOT11]{FOT}
Masatoshi Fukushima, Yoichi Oshima, and Masayoshi Takeda.
\newblock {\em Dirichlet forms and symmetric {M}arkov processes}, volume~19 of
  {\em De Gruyter Studies in Mathematics}.
\newblock Walter de Gruyter \& Co., Berlin, extended edition, 2011.

\bibitem[KLSS19]{KLSS20}
Matthias Keller, Daniel Lenz, Marcel Schmidt, and Michael Schwarz.
\newblock Boundary representation of {D}irichlet forms on discrete spaces.
\newblock {\em J. Math. Pures Appl. (9)}, 126:109--143, 2019.

\bibitem[LM72]{LM72}
J.-L. Lions and E.~Magenes.
\newblock {\em {Non-homogeneous boundary value problems and applications. Vol.
  I}}.
\newblock Die Grundlehren der mathematischen Wissenschaften, Band 181.
  Springer-Verlag, New York-Heidelberg, 1972.
\newblock Translated from the French by P. Kenneth.

\bibitem[Ouh96]{Ouh96}
E.~Ouhabaz.
\newblock Invariance of closed convex sets and domination criteria for
  semigroups.
\newblock {\em Potential Analysis}, 5(6):611--625, 1996.

\bibitem[Pos14]{Pos15}
Andrea Posilicano.
\newblock Markovian extensions of symmetric second order elliptic differential
  operators.
\newblock {\em Math. Nachr.}, 287(16):1848--1885, 2014.

\bibitem[RS78]{RS78}
Michael Reed and Barry Simon.
\newblock {\em Methods of modern mathematical physics. {IV}. {A}nalysis of
  operators}.
\newblock Academic Press [Harcourt Brace Jovanovich, Publishers], New
  York-London, 1978.

\bibitem[{Sch}17]{S17}
Marcel {Schmidt}.
\newblock {\em {Energy forms}}.
\newblock PhD thesis, Mar 2017.

\bibitem[Sch20a]{S18}
Marcel Schmidt.
\newblock A note on reflected {D}irichlet forms.
\newblock {\em Potential Anal.}, 52(2):245--279, 2020.

\bibitem[{Sch}20b]{Sch20}
Michael {Schwarz}.
\newblock {\em {Nodal Domains and Boundary Representation for DirichletForms}}.
\newblock PhD thesis, Jan 2020.

\bibitem[Sto92]{Sto92}
Peter Stollmann.
\newblock Smooth perturbations of regular {D}irichlet forms.
\newblock {\em Proc. Amer. Math. Soc.}, 116(3):747--752, 1992.

\bibitem[Sto93]{St93}
Peter Stollmann.
\newblock Closed ideals in {D}irichlet spaces.
\newblock {\em Potential Anal.}, 2(3):263--268, 1993.

\bibitem[SV96]{SV96}
Peter Stollmann and J\"{u}rgen Voigt.
\newblock Perturbation of {D}irichlet forms by measures.
\newblock {\em Potential Anal.}, 5(2):109--138, 1996.

\end{thebibliography}
\bibliographystyle{alpha}
\end{document}